\documentclass[a4paper,12pt, twoside, pdftex]{amsart}

\usepackage{fullpage}
\usepackage{amsmath, amsthm, amssymb}
\usepackage{txfonts}
\usepackage{amsfonts}
\usepackage{prettyref}
\usepackage{mathrsfs}
\usepackage[all]{xy}
\usepackage[utf8]{inputenc}

\newtheorem{dummy}{dummy}[section]
\newtheorem{lemma}[dummy]{Lemma}
\newtheorem{theorem}[dummy]{Theorem}

\newtheorem{corollary}[dummy]{Corollary}

\theoremstyle{definition}
\newtheorem{definition}[dummy]{Definition}
\newtheorem*{definition*}{Definition}

\newtheorem{remark}[dummy]{Remark}

\newtheorem*{acknowledgements}{Acknowledgements}

\numberwithin{equation}{section}

\usepackage{tikz}
\usetikzlibrary{backgrounds}

\newrefformat{th}{Theorem~\ref{#1}}
\newrefformat{cr}{Corollary~\ref{#1}}
\newrefformat{lm}{Lemma~\ref{#1}}
\newrefformat{dl}{Definition-Lemma~\ref{#1}}
\newrefformat{df}{Definition~\ref{#1}}
\newrefformat{cl}{Claim~\ref{#1}}
\newrefformat{sl}{Sublemma~\ref{#1}}
\newrefformat{pr}{Proposition~\ref{#1}}
\newrefformat{cj}{Conjecture~\ref{#1}}
\newrefformat{st}{Step~\ref{#1}}
\newrefformat{sc}{Section~\ref{#1}}
\newrefformat{df}{Definition~\ref{#1}}
\newrefformat{rm}{Remark~\ref{#1}}
\newrefformat{q}{Question~\ref{#1}}
\newrefformat{pb}{Problem~\ref{#1}}
\newrefformat{cd}{Condition~\ref{#1}}
\newrefformat{example}{Example~\ref{#1}}
\newrefformat{he}{Heore~\ref{#1}}
\newrefformat{fg}{Figure~\ref{#1}}
\newrefformat{tb}{Table~\ref{#1}}
\newrefformat{as}{Assumption~\ref{#1}}

\newcommand{\pref}{\prettyref}

\newcommand{\colim}{\operatorname{colim}}
\newcommand{\cont}{\operatorname{cont}}
\newcommand{\Coh}{\operatorname{Coh}}

\newcommand{\Core}{\operatorname{Core}}

\newcommand{\del}{\partial}

\newcommand{\DG}{\operatorname{DG}}

\newcommand{\Exit}{\operatorname{Exit}}

\newcommand{\Fan}{\operatorname{Fan}}
\newcommand{\Fuk}{\operatorname{Fuk}}
\newcommand{\fsh}{\operatorname{fsh}}

\newcommand{\Hom}{\operatorname{Hom}}

\newcommand{\msh}{\operatorname{\mu sh}}

\newcommand{\Nbd}{\operatorname{Nbd}}

\newcommand{\ret}{\operatorname{ret}}

\newcommand{\Sk}{\operatorname{Sk}}
\newcommand{\Sh}{\operatorname{Sh}}
\newcommand{\sstar}{\operatorname{star}}

\newcommand{\UI}{\operatorname{UI}}

\newcommand{\cL}{\mathcal{L}}
\newcommand{\cM}{\mathcal{M}}

\newcommand{\cT}{\mathcal{T}}

\newcommand{\cW}{\mathcal{W}}

\newcommand{\bA}{\mathbb{A}}
\newcommand{\bC}{\mathbb{C}}

\newcommand{\bL}{\mathbb{L}}

\newcommand{\bR}{\mathbb{R}}

\newcommand{\bZ}{\mathbb{Z}}

\newcommand{\bfT}{\mathbf{T}}

\newcommand{\bfW}{\mathbf{W}}

\newcommand{\frakf}{\mathfrak{f}}

\begin{document}

\title{Sectorial covers over fanifolds}
\author[H.~Morimura]{Hayato Morimura}
\address{
Kavli Institute for the Physics and Mathematics of the Universe (WPI),
University of Tokyo,
5-1-5 Kashiwanoha,
Kashiwa,
Chiba,
277-8583,
Japan.}
\email{hayato.morimura@ipmu.jp}

\date{}
\pagestyle{plain}

\begin{abstract}
For the stopped Weinstein sector associated with any fanifold introduced in
\cite{GS1}
by Gammage--Shende,
we construct a Weinstein sectorial cover
which allows us to describe homological mirror symmetry over the fanifold as an isomorphism of cosheaves of categories.
In a special case,
our Weinstein sectorial cover gives a lift of the open cover for the global skeleton of a very affine hypersurface computed in
\cite{GS2}.
\end{abstract}

\maketitle

\section{Introduction}
Sectorial descent established in
\cite{GPS2}
by Ganatra--Pardon--Shende is one of the local-to-global behaviors of Fukaya categories.
It enables us to compute the wrapped Fukaya category
$\cW(X)$
of a Weinstein sector
$X$
by gluing that of local pieces,
provided its Weinstein sectorial cover
$X = \bigcup^n_{i = 1} X_i$.
In other words,
wrapped Fukaya categories are cosheaves with respect to Weinstein sectorial covers.
The result is valid also for partially wrapped Fukaya categories of stopped Weinstein sectors with mostly Legendrian stops away from boundaries.

Considering its computational power,
given a Weinstein sector,
it is natural to seek its Weinstein sectorial cover. 
Moreover,
when working with homological mirror symmetry(HMS),
one often needs such a cover to be compatible with that of its mirror. 
So far,
there are not so many interesting examples of Weinstein sectorial covers
\cite{CKL, GJ, GL}.
One reason for this might be easier access to other local-to-global behaviors
\cite{GS2, GS1, Lee, MSZ, PS1, PS2, PS3}.
Although they might work to solve a given problem,
construction of Weinstein sectorial covers still should be a fundamental geometric question.

In this paper,
we give a Weinstein sectorial cover of the stopped Weinstein sector
$\widetilde{\bfW}(\Phi)$
associated with a fanifold
$\Phi$.
A fanifold is introduced in
\cite{GS1}
by Gammage--Shende to discuss HMS for more general spaces than toric stacks on the $B$-side.
It is a stratified space obtained by gluing rational polyhedral fans of cones,
which provides the organizing
topological
and
discrete
data for HMS at large volume. 
To
$\Phi$
they associated an algebraic space
$\bfT(\Phi)$
obtained as the gluing of the toric varieties
$T_\Sigma$
for the fans
$\Sigma$.
Based on an idea from SYZ fibrations,
they constructed its mirror
$\widetilde{\bfW}(\Phi)$
by inductive Weinstein handle attachments.

The guiding principle behind their construction of
$\widetilde{\bfW}(\Phi)$
was that
the FLTZ Lagrangians
$\bL(\Sigma)$
and
the projections
$\bL(\Sigma) \to \Sigma$
should glue to yield
its relative skeleton 
$\widetilde{\bL}(\Phi)$
and
a map
$\pi \colon \widetilde{\bL}(\Phi) \to \Phi$.
As predicted in
\cite[Remark 4.5]{GS1},
the canonical lifts of the projections
$\bL(\Sigma) \to \Sigma$
glue to yield a version of $A$-side SYZ fibration
$\bar{\pi} \colon \widetilde{\bfW}(\Phi) \to \Phi$
after some modifications
\cite[Theorem 1.2]{Mor}. 
The map
$\bar{\pi}$
enables us to lift any cover of
$\Phi$
to
$\widetilde{\bfW}(\Phi)$.
Inspired by the idea in
\cite[Example 1.34]{GPS2}
and
using this aspect of
$\bar{\pi}$,
after some modification,
we construct a Weinstein sectorial cover of
$\widetilde{\bfW}(\Phi)$
with the desired compatibility in Section
$3$. 

\begin{theorem} \label{thm:main}
There exists a Weinstein sectorial cover of
$\widetilde{\bfW}(\Phi)$.
Moreover,
it restricts to an open cover of the relative skeleton
$\widetilde{\bL}(\Phi)$
compatible with a cover of
$\bfT(\Phi)$.
\end{theorem}

When
$\Phi$
is the fanifold from
\cite[Example 4.23]{GS1},
$\widetilde{\bL}(\Phi)$
coincides with the skeleton
$\Core(H)$
\cite[Theorem 6.12]{GS2}
of a very affine hypersurface
$H$.
In Section
$4$,
we show that our cover gives a lift of the open cover
\cite[Corollary 4.8]{GS2}
of
$\Core(H)$,
giving an affirmative answer to the conjecture raised in
\cite[Section 2.5]{GS2}
and
making the heuristic of
\cite[Section 2.1]{GS2}
completely rigorous.

Via sectorial descent we obtain a description of HMS for
$(\widetilde{\bfW}(\Phi), \bfT(\Phi))$
as an isomorphism of cosheaves of categories.
See Section
$3.4$
for the precise definitions of the associated cosheaves.
The isomorphism can be regarded as a dual to
\cite[Theorem 5.3]{GS1}
and
encodes the gluing of coherent constructible correspondence
\cite{Kuw}
for local pieces. 

\begin{corollary}[\pref{thm:HMS}]
There is an isomorphism
$\bar{\pi}_* \Fuk_*
\cong
\Coh_! \circ \bfT$
of cosheaves of categories over
$\Phi$
whose global section yields an equivalence
$\Fuk(\widetilde{\bfW}(\Phi))
\simeq
\Coh(\bfT(\Phi))$.
\end{corollary}

Our construction of the cover is close to the more general strategy suggested in
\cite{BC}
by Bai--Côté.
When
$X = T^* M$,
they triangulate
$M$
and
cover
$X$
by the cotangent bundles
$T^* \sstar(v_\alpha)$
of the stars of the vertices
$v_\alpha$.
When
$X$
is a polarizable Weinstein manifold,
they deform
$X$
to another Weinstein manifold with arboreal skeleton
which by
\cite[Section 4.2]{BC}
admits
a canonical Whitney stratification,
in particular,
a triangulation due to Goresky.
Cover the deformation of
$X$
by arboreal thickenings of the stars of the vertices. 
Unfortunately,
there are obstructions to making their strategy rigorous
which involves developing a good theory of arboreal sectors.
 
Another general strategy might come from simplicial decompositions
introduced in
\cite{Asp}
by Asplund.
By
\cite[Theorem 1.4]{Asp}
a simplicial decomposition of a Weinstein manifold
$X$
bijectively corresponds to its sectorial cover
which is good in the sense of
\cite[Section 1]{Asp}.
Since the convex completion of
$\widetilde{\bfW}(\Phi)$
is an ordinary Weinstein manifold,
by
\cite[Lemma 3.13]{Asp}
one could refine our Weinstein sectorial cover to be a good sectorial cover of the convex completion.
It would be an interesting problem to find a simplicial decomposition of a Weinstein manifold
corresponding to a Weinstein sectorial cover
and
compatible with the cover of its given mirror.

\begin{acknowledgements}
The author was supported by SISSA PhD scholarships in Mathematics.
He is grateful to an anonymous referee for careful reading and helpful comments.
This work was partially supported by
JSPS KAKENHI Grant Number
JP23KJ0341.  
\end{acknowledgements}

\section{Large volume limit fibrations} 
In this section,
after the construction of the stopped Weinstein sector
$\widetilde{\bfW}(\Phi)$
associated with a fanifold
$\Phi$
introduced in
\cite{GS1}
by Gammage--Shende,
we review that of the map
$\bar{\pi} \colon \widetilde{\bfW}(\Phi) \to \Phi$
from
\cite[Theorem 1.2]{Mor}
restricting to the map
$\pi \colon \widetilde{\bL}(\Phi) \to \Phi$
from the relative skeleton
$\widetilde{\bL}(\Phi)$.
The stopped Weinstein sector
$\widetilde{\bfW}(\Phi)$
is obtained by inductively attaching products of the cotangent bundles of
real tori
and
strata of
$\Phi$.
Each step requires us to modify Weinstein structures near gluing regions.

\subsection{Fanifolds}
Throughout the paper,
we will work with only those stratified spaces
$\Phi$
which satisfy the following conditions:

\begin{itemize}
\item[(i)]
$\Phi$
has finitely many strata.
\item[(ii)]
$\Phi$
is conical in the complement of a compact subset
\cite[Page 37]{GS1}. 
\item[(iii)]
$\Phi$
is given as a germ of a closed subset in an ambient manifold
$\cM$.
\item[(iv)]
The strata of
$\Phi$
are smooth submanifolds of
$\cM$.
\item[(v)]
The strata of
$\Phi$
are contractible.
\end{itemize}
We will express properties of
$\Phi$
in terms of
$\cM$
as long as they only depend on the germ of
$\Phi$.
Taking the normal cone
$C_S \Phi \subset T_S \cM$
for each stratum
$S \subset \Phi$,
one obtains a stratification on
$C_S \Phi$
induced by that of a sufficiently small tubular neighborhood
$T_S \cM \to \cM$.

\begin{definition}
The stratified space
$\Phi$
is
\emph{smoothly normally conical}
if for each stratum
$S \subset \Phi$
some choice of tubular neighborhood
$T_S \cM \to \cM$
induces locally near
$S$
a stratified diffeomorphism
$C_S \Phi \to \Phi$,
which in turn induces the identity
$C_S \Phi \to C_S \Phi$.
\end{definition}

We write
$\Fan^\twoheadrightarrow$
for the category whose objects are pairs
$(M, \Sigma)$
of a laticce
$M$
and
a stratified space
$\Sigma$
by finitely many rational polyhedral cones in
$M_\bR = M \otimes_\bZ \bR$.
For any
$(M, \Sigma), (M^\prime, \Sigma^\prime) \in \Fan^\twoheadrightarrow$
a morphism
$(M, \Sigma) \to (M^\prime, \Sigma^\prime)$
is given by
the data of a cone
$\sigma \in \Sigma$
and
an isomorphism
$M / \langle \sigma \rangle \cong M^\prime$
such that
$\Sigma^\prime
=
\Sigma / \sigma
=
\{ \tau / \langle \sigma \rangle \subset M_\bR / \langle \sigma \rangle \  | \ \tau \in \Sigma, \sigma \subset \bar{\tau} \}$.
We denote by
$\Fan^\twoheadrightarrow_{\Sigma /}$
for
$(M, \Sigma) \in \Fan^\twoheadrightarrow$
the full subcategory of objects
$(M^\prime, \Sigma^\prime)$
with
$\Sigma^\prime = \Sigma / \sigma$
for some
$\sigma \in \Sigma$.

\begin{definition}[{\cite[Definition 2.4]{GS1}}] \label{dfn:fanifold}
A
\emph{fanifold}
is a smoothly normally conical stratified space 
$\Phi \subset \cM$
equipped with the following data:
\begin{itemize}
\item
A functor
$\Exit(\Phi) \to \Fan^\twoheadrightarrow$
from the exit path category of
$\Phi$
whose value on each stratum
$S$
is a pair
$(M_S, \Sigma_S)$
of a lattice
$M_S$
and
a rational polyhedral fan
$\Sigma_S \subset M_{S, \bR}$
called the
\emph{associated normal fan}.
\item
For each stratum
$S \subset \Phi$
a trivialization
$\phi_S \colon T_S \cM \cong M_{S, \bR}$
of the normal bundle carrying the induced stratification on
$C_S \Phi$
to the standard stratification induced by
$\Sigma_S$.
\end{itemize}
These data must make the diagram
\begin{align*}
\begin{gathered}
\xymatrix{
T_S \cM \ar[r]^-{\phi_S} \ar_{}[d] & M_{S, \bR} \ar^{}[d] \\
T_{S^\prime} \cM |_S \ar[r]_-{\phi_{S^\prime}} & M_{S^\prime, \bR},
}
\end{gathered}
\end{align*}
commute for any stratum
$S^\prime$
of the induced stratification on
$\Nbd(S)$,
where the left vertical arrow is the quotient by the span of
$S^\prime$.
The right vertical arrow corresponds to the map
$M_S \to M_{S^\prime}$ 
on lattices.
\end{definition}

\begin{remark}
One can identify the exit path category
$\Exit(\Sigma)$
with
$\Fan^\twoheadrightarrow_{\Sigma /}$
as a poset via
$\sigma \mapsto [\Sigma \mapsto \Sigma / \sigma]$.
In addition,
the normal geometry to
$\sigma$
is the geometry of
$\Sigma / \sigma$.
Due to the condition
(v),
$\Exit_S(\Phi)$
is equivalent to the poset
$\Exit(\Sigma_S)$,
the full subcategory of exit paths starting at
$S$
contained inside a sufficiently small neighborhood of
$S$.
\end{remark}

\begin{lemma}[{\cite[Remark 2.12]{GS1}}] \label{lem:filtration}
Any fanifold
$\Phi$
admits a filtration
\begin{align} \label{eq:filtration1}
\Phi_0
\subset
\Phi_1
\subset
\cdots
\subset
\Phi_n = \Phi, \
n = \dim \Phi
\end{align}
where
$\Phi_k$
are subfanifolds defined as sufficiently small neighborhoods of $k$-skeleta
$\Sk_k(\Phi)$,
the closure of the subset of $k$-strata.
\end{lemma}
\begin{proof}
Suppose that
$\Phi_{k-1}$
has the desired filtration
$\Phi_0 \subset \Phi_1 \subset \cdots \subset \Phi_{k-1}$,
where
$\Phi_0$
is the disjoint union
$\bigsqcup_P \Sigma_P$
for all $0$-strata
$P \subset \Phi$
equipped with the canonically induced fanifold structure.
The ideal boundary
$\del_\infty S$
might have some subset
$\del_{in} S$
which is in the direction to the interior of
$\Phi$.
Perform the gluing
$\Phi_{k-1} \#_{(\Sigma_S \times \del_{in} S)} (\Sigma_S \times S)$
which
equals
$\Sigma_S \times S$
when
$\Phi_{k-1}$
is empty
and
equals
$\Phi_{k-1}$
unless
$S$
is an
\emph{interior}
$k$-stratum,
i.e.
$\del_{in} S = \del S_\circ$
where
$S_\circ = S \setminus \Phi_{k-1}$
is a manifold-with-boundary.
Let
\begin{align*}
\Phi_k = \Phi_{k-1} \#_{\bigsqcup_S (\Sigma_S \times \del_{in} S)} \bigsqcup_S (\Sigma_S \times S)
\end{align*}
be the result of such gluings for all $k$-strata
$S \subset \Phi$.
Then
$\Phi_k \subset \Phi$
is a subfanifold containing
$\Sk_k(\Phi)$
since the products
$\Sigma_S \times S$
are canonically fanifolds.
\end{proof}

\subsection{Weinstein handle attachments}
Throughout the paper,
we work with Liouville manifolds of finite type
so that
any Liouville manifold
$X$
will be the completion of some Liouville domain
$X_0$.
Then its skeleton
$\Core(X)$
will always be compact.
We use the symbol
$\del_\infty X$
to denote both
the ideal boundary of
$X$
and
the actual boundary of
$X_0$
to preserve the symbols
$\del X_0, \del X$.
All Liouville manifolds of our interest are Weinstein.
Namely,
the Liouville vector field
$Z$
of each
$X$
is gradient-like for a Morse--Bott function
$\phi \colon X \to \bR$
which is constant on the cylindrical end
$(\bR_{t \geq 0} \times \del_\infty X, e^t (\lambda |_{\del_\infty X}))$.
Hence
$\Core(X)$
will also be isotropic by
\cite[Lemma 11.13(a)]{CE}.

Let
$W$
be a Weinstein domain with a smooth Legendrian
$\cL \subset \del_\infty W$.
Fix a standard neighborhood of
$\cL$
in
$\del_\infty W$
\begin{align*}
\eta
\colon
\Nbd_{\del_\infty W}(\cL)
\hookrightarrow
J^1 \cL
=
T^* \cL \times \bR
\end{align*}
which extends to a neighborhood in the Weinstein domain
$W$
\begin{align*}
\xi
\colon
\Nbd_W(\cL)
\hookrightarrow
J^1 \cL \times \bR_{\leq 0}
\cong
T^*(\cL \times \bR_{\leq 0}),
\end{align*}
where the Liouville flow on
$W$
gets identified with the translation action on
$\bR_{\leq 0}$.
By
\cite[Proposition 2.9]{Eli}
one can modify the Weinstein structure near
$\cL$
so that
the Liouville flow gets identified with the cotangent scaling on
$T^*(\cL \times \bR_{\leq 0})$.
We denote by
$\widetilde{W}$
its conic completion,
which is a Weinstein sector by
\cite[Lemma 2.13]{GPS1}.
Then
$\eta$
induces a neighborhood 
$\widetilde{\eta}
\colon
\Nbd_{\del_\infty \widetilde{W}}(\cL)
\hookrightarrow
J^1 \cL$
which extends to a neighborhood
$\widetilde{\xi}
\colon
\Nbd_{\widetilde{W}}(\cL)
\hookrightarrow
J^1 \cL \times \bR_{\leq 0}$.

Given Weinstein domains
$W, W^\prime$
with smooth Legendrian embeddings
$\del_\infty W \hookleftarrow \cL \hookrightarrow \del_\infty W^\prime$,
we write
$W \#_\cL W^\prime$
for the gluing
\begin{align*}
W \#_{\eta^{-1}(T^* \cL \times \{ 0 \}) \cong \eta^{\prime -1}(T^* \cL \times \{ 0 \})} W^\prime
=
\widetilde{W} \cup_{J^1 \cL} \widetilde{W^\prime}
\end{align*}
which yields a Weinstein sector with skeleton
\begin{align*}
\Core(W \#_\cL W^\prime)
=
\Core(W, \cL)
\cup_\cL
\Core(W^\prime, \cL)
\end{align*}
where
$\Core(W, \cL)$
denotes the relative skeleton of the stopped Weinstein domain
$(W, \cL)$,
i.e.
the disjoint union
$\Core(W) \sqcup \bR \cL \subset W$
of
$\Core(W)$
and the saturation of
$\cL$
by the Liouville flow.

\begin{definition}[{\cite[Definition 4.9, Lemma 4.10]{GS1}}]
Let
$W$
be a Weinstein domain with a smooth Legendrian
$\cL \subset \del_\infty W$.
Then a Lagrangian
$\bL \subset W$
is
\emph{biconic}
along
$\cL$
if and only if
it is conic
and
there exists some
$\xi
\colon
\Nbd_W(\cL)
\hookrightarrow
T^*(\cL \times \bR_{\leq 0})$
such that
\begin{align*}
\xi(\bL \cap \Nbd_W(\cL))
\subset
T^* \cL \times \{ 0 \} \times \bR_{\leq 0}
\subset
J^1 \cL \times \bR_{\leq 0}
\cong
T^*(\cL \times \bR_{\leq 0}),
\end{align*}
where the Liouville flow on
$W$
gets identified with the translation action on
$\bR_{\leq 0}$.
\end{definition}

Given biconic Lagrangians
$\bL \subset W, \bL^\prime \subset W^\prime$
with matching ends
\begin{align*}
\widetilde{\eta}(\bL \cap \Nbd_{\del_\infty W}(\cL))
=
\widetilde{\eta}^\prime(\bL^\prime \cap \Nbd_{\del_\infty W^\prime}(\cL))
\end{align*}
in
$W \#_\cL W^\prime$,
we write
$\bL \#_\cL \bL^\prime$
for the gluing
$\widetilde{\bL}
\cup_{\widetilde{\eta}(\bL \cap \Nbd_{\del_\infty W}(\cL))}
\widetilde{\bL^\prime}$
in
$W \#_\cL W^\prime$,
where
$\widetilde{\bL} \subset \widetilde{W},
\widetilde{\bL^\prime} \subset \widetilde{W^\prime}$
denote the saturations of
$\bL, \bL^\prime$
under the Liouville flows.
Since any biconic subsets in
$W, W^\prime$
remain conic in
$W \#_\cL W^\prime$,
the gluing
$\bL \#_\cL \bL^\prime$
is a conic Lagrangian.

For a closed manifold
$\widehat{M}$
and
a compact manifold-with-boundary
$S$,
consider the Weinstein manifold-with-boundary
$\widetilde{W^\prime} = T^* \widehat{M} \times T^* S$
with a smooth Legendrian
$\cL = \widehat{M} \times \del S$
taken to be a subset of the zero section.
Here,
we equip
$T^* S$
the tautological Liouville structure
so that 
the Liouville vector field is the generator of fiberwise radial dilation.
Then the Liouville flow on
$\widetilde{W^\prime}$
near
$\cL$
is the cotangent scaling.
We write
$W^\prime$
for a domain completing to
$\widetilde{W^\prime}$.
One can check that
the above gluing procedure carries over,
although
$\cL$
does not belong to
$\del_\infty W^\prime$.

\begin{remark} \label{rmk:2.3}
As explained in
\cite[Example 2.3]{GPS1},
when
$S$
is noncompact
$T^* S$
equipped with the tautological Liouville form is not even a Liouville manifold.
However,
if
$S$
is the interior of a compact manifold-with-boundary,
then
$T^* S$
becomes a Liouville manifold after a suitable modification of the Liouville form near the boundary making it convex.
We
use the same symbol
$T^* S$
to denote the result
and
regard it as a stopped Weinstein sector.
\end{remark}

\begin{definition}[{\cite[Definition 4.11]{GS1}}] \label{dfn:extension}
Let
$W$
be a Weinstein domain with a smooth Legendrian embedding
$\widehat{M} \times \del S \hookrightarrow \del_\infty W$.

(1)
A
\emph{handle attachment}
is the gluing
$W \#_{\widehat{M} \times \del S} W^\prime$
respecting the product structure.

(2)
If
$\bL \subset W$
is a biconic subset along
$\widehat{M} \times \del S$
which locally factors as
\begin{align*}
\eta(\bL \cap \Nbd_{\del_\infty W}(\widehat{M} \times \del S))
=
\bL_S \times \del S
\subset
T^*\widehat{M} \times T^* \del S
\end{align*}
for some fixed conic Lagrangian
$\bL_S$,
then the
\emph{extension}
of
$\bL$
through the handle is the gluing
$\bL \#_{\widehat{M} \times \del S} (\bL_S \times S)$
respecting the product structure.
\end{definition}

Let
$\Sigma \subset M_\bR$
be a rational polyhedral fan
and
$\widehat{M}$
the real $n$-torus
$\Hom(M, \bR / \bZ) = M^\vee_\bR / M^\vee$.
Recall from
\cite[Section 3.1]{FLTZ12}
that
the
\emph{FLTZ Lagrangian}
is the union
\begin{align*}
\bL(\Sigma)
=
\bigcup_{\sigma \in \Sigma} \bL_\sigma
=
\bigcup_{\sigma \in \Sigma} \sigma^\perp \times \sigma
\end{align*}
of conic Lagrangians
$\sigma^\perp \times \sigma
\subset
\widehat{M} \times M_\bR
\cong
T^*\widehat{M}$,
where
$\sigma^\perp$
is the real $(n - \dim \sigma)$-subtorus
\begin{align*}
\{ x \in \widehat{M} \ | \ \langle x, v \rangle = 0 \ \text{for all } v \in \sigma \}.
\end{align*}

\begin{lemma}[{\cite[Lemma 4.16]{GS1}}] \label{lem:coordinates}
Let
$\Sigma \subset M_\bR$
be a rational polyhedral fan.
Then near the Legendrian
$\del_\infty \bL_\sigma \subset \del_\infty T^* \widehat{M}$
for each cone
$\sigma \in \Sigma$
there are local coordinates
\begin{align*}
\eta_\sigma
\colon
\Nbd(\del_\infty \bL_\sigma)
\hookrightarrow
J^1 \del_\infty \bL_\sigma = T^* \del_\infty \bL_\sigma \times \bR 
\end{align*}
with the following properties.
\begin{itemize}
\item[(1)]
The Lagrangian
$\bL_\sigma$
is biconic along
$\del_\infty \bL_\sigma$.
\item[(2)]
For any nonzero cones
$\sigma, \tau \in \Sigma$
with
$\sigma \subset \bar{\tau}$,
we have
\begin{align} \label{eq:recursive}
\eta_\sigma (\del_\infty \bL_\tau \cap \Nbd(\del_\infty \bL_\sigma))
=
\bL_{\tau / \sigma} \times \del_\infty \sigma \times \{ 0 \}
\subset
T^* \sigma^\perp \times T^* \del_\infty \sigma \times \{ 0 \} 
=
T^* \del_\infty \bL_\sigma \times \{ 0 \}.
\end{align}
\item[(3)]
The FLTZ Lagrangian
$\bL(\Sigma)$
is biconic along each
$\del_\infty \bL_\sigma$.
\item[(4)]
For each
$\sigma \in \Sigma$
the local coordinates
$\eta_\sigma$
define a Weinstein hypersurface
\begin{align*}
R_\sigma = \eta^{-1}_\sigma (T^* \del_\infty \bL_\sigma \times \{ 0 \})
\subset
\del_\infty T^* \widehat{M}
\end{align*}
containing
$\del_\infty \bL_\sigma$
as its skeleton
and
$R_\tau \cap \Nbd(\del_\infty \bL_\sigma) \subset R_\sigma$
for any
$\tau \in \Sigma$
with
$\sigma \subset \bar{\tau}$.
\end{itemize}
\end{lemma}


\subsection{Construction}
The stopped Weinstein sector
$\widetilde{\bfW}(\Phi)$
is constructed inductively together with
$\widetilde{\bL}(\Phi)$
and
$\pi$
along the filtration
\pref{eq:filtration1}
of
$\Phi$.
When
$k = 0$,
let
$\widetilde{\bfW}(\Phi_0)
=
\bigsqcup_P T^* \widehat{M}_P, \
\widetilde{\bL}(\Phi_0)
=
\bigsqcup_P \bL(\Sigma_P)$
where each
$T^* \widehat{M}_P$
is equipped with the canonical Liouville structure.
We fix an identification
$T^* \widehat{M}_P \cong \widehat{M}_P \times M_{P, \bR}$
to regard each
$\bL(\Sigma_P)$
as a conic Lagrangian submanifold of
$T^* \widehat{M}_P$.
Let
$\pi_0 \colon \widetilde{\bL}(\Phi_0) \to \Phi_0$
be the map induced by the projection to cotangent fibers.
Then the triple
$(\widetilde{\bfW}(\Phi_0), \widetilde{\bL}(\Phi_0), \pi_0)$
satisfies the conditions
(1), \ldots, (4).
Define
$\bar{\pi}_0$
as the composition
$\ret_0 \circ \tilde{\pi}_0$,
where
$\tilde{\pi}_0$
denotes the disjoint union of the projections to the cotangent fibers
$T^* \widehat{M}_P
\to
M_{P, \bR}$
and
$\ret_0
\colon
\bigsqcup_P M_{P, \bR}
\to
\bigsqcup_P \Sigma_P$
denotes the disjoint union of maps induced by retractions.
See
\cite[Section 5.1]{Mor}
for details.
The retractions cut off the part of the image of
$\tilde{\pi}_0$
outside
$\Phi$.

Suppose that
we have constructed the triple
$(\widetilde{\bfW}(\Phi_{k-1}), \widetilde{\bL}(\Phi_{k-1}), \pi_{k-1})$
for the subfanifold
$\Phi_{k-1}$.
Let
$\cL_k$
be the disjoint union of
\begin{align*}
\cL^k_ {S^{(l)}}
=
\pi^{-1}_{k-1}(S^{(l)}) \cap \del_\infty \widetilde{\bL}(\Phi_{k-1})
\end{align*}
for all interior $l$-strata
$S^{(l)} \subset \Phi, k \leq l$.
There are smooth Legendrian embeddings
\begin{align*}
\del_\infty \bfW(\Phi_{k-1})
\hookleftarrow
\cL^k_ {S^{(l)}}
\hookrightarrow
\del \bigsqcup_{S^{(k)}} (T^* \widehat{M}_{S^{(k)}} \times T^* S^{(k)}_\circ). 
\end{align*}
Due to
the chart
\cite[Theorem 4.1(1)]{GS1}
for the previous step
and
\pref{lem:coordinates}(4),
we may define
$\widetilde{\bfW}(\Phi_k)$
as the handle attachment
\begin{align*}
\bfW(\Phi_{k-1}) \#_{\cL_k}  \bigsqcup_{S^{(k)}} (T^* \widehat{M}_{S^{(k)}} \times T^* S^{(k)}_\circ).
\end{align*}

There is a standard neighborhood
$\eta_k
\colon
\Nbd_{\del_\infty \bfW(\Phi_{k-1})}(\cL_k)
\hookrightarrow
J^1 \cL_k$
near
$\cL_k$
for which
$\widetilde{\bL}(\Phi_{k-1})$
is biconic along
$\cL_k$
and
$\del_\infty \widetilde{\bL}(\Phi_{k-1})$
locally factors as
\begin{align*}
\eta_k(\widetilde{\bL}(\Phi_{k-1}) \cap \Nbd_{\del_\infty \bfW(\Phi_{k-1})}(\cL_k))
=
\bigsqcup_{S^{(k)}}
(\bL(\Sigma_{S^{(k)}}) \times \del S^{(k)}_\circ).
\end{align*}
We define
$\widetilde{\bL}(\Phi_k)$
as the extension through the disjoint union of the handles
$T^* \widehat{M}_{S^{(k)}} \times T^* S^{(k)}_\circ$
\begin{align*}
\widetilde{\bL}(\Phi_{k-1}) \#_{\cL_k}  \bigsqcup_{S^{(k)}} (\bL(\Sigma_{S^{(k)}}) \times S^{(k)}_\circ).
\end{align*}
Let
$\pi_k \colon \widetilde{\bL}(\Phi_k) \to \Phi_k$
be the map induced by
$\pi_{k-1}$
and
the projections from
$T^* \widehat{M}_{S^{(k)}} \times T^* S^{(k)}_\circ$
to
the cotangent fiber direction in
$T^* \widehat{M}_{S^{(k)}}$
and
the base direction in
$T^* S^{(k)}_\circ$.
The triple
$(\widetilde{\bfW}(\Phi_k), \widetilde{\bL}(\Phi_k), \pi_k)$
satisfies the conditions
(1), \ldots, (4).

Here,
we explain how to construct
$\bar{\pi}_k$
when attaching the handle
$T^* \widehat{M}_{S^{(k)}} \times T^* S^{(k)}_\circ$
to
$\widetilde{\bfW}(\Phi_{k-1})$
for a single interior $k$-stratum
$S^{(k)} \subset \Phi$
with
$k \leq \dim \Phi$.
Define
$\tilde{\pi}_k$
as the map canonically induced by
$\tilde{\pi}_{k-1}, \tilde{\pi}_{0, S^{(k)}}$
and
the projection
$T^* S^{(k)}_\circ \to S^{(k)}_\circ$
to the base,
where
$\tilde{\pi}_{0, S^{(k)}}
\colon
T^* \widehat{M}_{S^{(k)}} \to M_{S^{(k)}, \bR}$
denotes the projection to the cotangent fibers.
Here,
we precompose the contraction
\begin{align} \label{eq:cont}
\cont_k \colon \widetilde{\bfW}(\Phi_k) \to \widetilde{\bfW}(\Phi_k)
\end{align}
of the cylindrical ends along the negative Liouville flow to the part of
$\del_\infty \bfW(\Phi_{k-1})$
intersecting
$\del T^* S^{(k)}_\circ$.
Let
$\ret_k \colon M_{S^{(k)}, \bR} \to \Sigma_{S^{(k)}}$
be the map induced by a retraction defined similarly to
$\ret_0$.
Then
$\bar{\pi}_k$
is the map canonically induced by
$\bar{\pi}_{k-1},
\bar{\pi}_{0, S^{(k)}} = \ret_k \circ \tilde{\pi}_{0, S^{(k)}}$
and
the projection
$T^* S^{(k)}_\circ \to S^{(k)}_\circ$
to the base.

\section{Weinstein sectorial covers}
In this section,
we construct a Weinstein sectorial cover of
$\widetilde{\bfW}(\Phi)$
by lifting a suitable cover of
$\Phi$
along
$\bar{\pi}$.
The idea is similar to the construction for cotangent bundles.
However,
here we need extra care to make it compatible with the  handle attachment process.
Via sectorial descent,
our cover gives rise to a constructible cosheaf
$\bar{\pi}_* \Fuk_*$
of categories over
$\Phi$
with respect to a certain natural topology.
We show that 
$\bar{\pi}_* \Fuk_*$
is isomorphic to the $B$-side counterpart
$\Coh_! \circ \bfT$.

\subsection{Sectorial covers}
Throughout the paper,
for a stopped Liouville sector
$(X, \frakf)$
we denote by
$\cW(X, \frakf)$
the partially wrapped Fukaya category
\cite[definition 2.31]{GPS2}.
When
$\frakf = \emptyset$,
it gives the wrapped Fukaya category
$\cW(X)$
of a Liouville sector
$X$
\cite[Definition 3.36]{GPS1}.
Due to
\cite{GPS2},
given a Liouville sector
$X$,
if it admits a Weinstein sectorial cover,
then the wrapped Fukaya category
$\cW(X)$
can be computed by gluing that of pieces in the cover.

\begin{definition}[{\cite[Definition 12.19, Remark 12.20]{GPS2}}]
Let
$X$
be a Liouville manifold-with-boundary.
Suppose that
$X$
admits a cover
$X = \bigcup^n_{i = 1} X_i$
by Liouville sectors-with-sectorial-corners
$X_i \subset X$
\cite[Definition 12.14]{GPS2},
with precisely two faces
$\del^1 X_i = X_i \cap \del X$
and
the point set topological boundary
$\del^2 X_i$
of
$X_i$
meeting along the corner locus
$\del X \cap \del^2 X_i = \del^1 X_i \cap \del^2 X_i$.
Such a cover is
\emph{sectorial}
if the collection of boundaries
$\del X, \del^2 X_1, \ldots, \del^2 X_n$
is sectorial in the sense of
\cite[Definition 12.2]{GPS2}.
We also allow
$X_i$
without corners,
requiring
$\del^2 X_i$
to be disjoint from
$\del X$
and
$\del X, \del^2 X_1, \ldots, \del^2 X_n$
to be sectorial
and
$\bigcup^n_{i = 1} X_i \hookrightarrow X$
to be a trivial inclusion.
\end{definition}

For any sectorial cover
$X = \bigcup^n_{i =1} X_i$,
stratify
$X$
by strata
\begin{align} \label{eq:strata}
X_{J, K, L}
=
\bigcup_{j \in J} X_j
\cap
\bigcap_{k \in K} \del X_k \
\setminus
\bigcap_{l \in L} X_l
\end{align}
ranging over
$J \sqcup K \sqcup L = \{ 1, \ldots, n \}$.
The closure of each
$X_{J, K, L}$
is a submanifold-with-corners,
whose symplectic reduction is a Liouville sector-with-sectorial-corners.
Recall that
a sectorial cover
$X = \bigcup^n_{i =1} X_i$
is
\emph{Weinstein}
\cite[Definition 12.15]{GPS2}
if the convex completions of all of the symplectic reductions of strata
\pref{eq:strata}
are Weinstein up to deformation.
For any Weinstein sectorial cover
$X = \bigcup^n_{i = 1} X_i$
of a Weinstein sector
$X$,
the induced functor
\begin{align} \label{thm:hcolimit}
\displaystyle \colim_{\emptyset \neq I \subset \{ 1, \ldots, n \}} 
\cW(\bigcap_{i \in I} X_i)
\to
\cW (X)
\end{align}
is a pretriangulated equivalence
\cite[Theorem 1.35]{GPS2}.

\subsection{Alternative description of $\widetilde{\bfW}(\Phi)$ and $\bar{\pi}$}
Recall that
$\widetilde{\bfW}(\Phi)$
is a stopped Weinstein sector obtained by Weinstein handle attachments.
See also Remark
\pref{rmk:2.3}.
The boundary comes from conic completion of the boundaries of the contact surroundings of
$T^* \cL_k$
in
$\del_\infty{\bfW}(\Phi_{k-1})$.
Below,
preserving a sufficiently small neighborhood of
$\widetilde{\bL}(\Phi)$,
we will modify the construction in order to obtain a stopped Weinstein manifold.
See
\cite[Figure 3.1]{Eli}
to get the picture of our modification.
As in the original construction,
we proceed by induction on
$k$.

Whenever we attach a handle
$T^* \widehat{M}_{S^{(k)}} \times T^* S^{(k)}_\circ$ 
to
$\widetilde{\bfW}(\Phi_{k-1})$,
deform
$\del_\infty \widetilde{\bfW}(\Phi_{k-1})$ 
so that the complement in
$\del_\infty \widetilde{\bfW}(\Phi_{k-1})$
of the contact surrounding
$T^* \bL(\Sigma_{S^{(k)}}) \times \del S^{(k)}_\circ$
gets connected smoothly to
$T^* \widehat{M}_{S^{(k)}} \times \del_\infty T^* S^{(k)}_\circ$.
Here,
we regard
$T^* S^{(k)}_\circ$
as a Liouville sector equipped with the tautological Liouville form.
Now,
conic completion yields a Weinstein manifold without boundary.
We use the same symbol
$\widetilde{\bfW}(\Phi)$
to denote the result,
as our modification preserves a sufficiently small neighborhood of
$\widetilde{\bL}(\Phi)$.
One can also modify the map
$\bar{\pi}$
in a straightforward way.
Namely,
when precomposing
\pref{eq:cont},
we contract the cylindrical ends of
$T^* \widehat{M}_{S^{(k)}} \times T^* S^{(k)}_\circ$
to
$T^* \widehat{M}_{S^{(k)}} \times \del_\infty T^* S^{(k)}_\circ$.

\subsection{Construction}
We construct a Weinstein sectorial cover of the modification
$\widetilde{\bfW}(\Phi)$
as the inverse image under
$\bar{\pi}$
of a certain cover of
$\Phi$.
For expository reasons,
we will assume that
$\Phi$
is closed.
Since each stratum is closed,
one can take its barycenter.
Connect by an edge inside
$\Phi$
the barycenter on each top dimensional stratum with that on $1$-dimensional lower strata adjacent to it.
Repeat the process inductively until reaching the barycenters of $1$-strata.
There might be some lower dimensional closed strata
which are not adjacent to higher dimensional strata.
Beginning the process also from such strata,
we obtain a partition of
$\Phi$
defined by
the additional edges connecting the barycenters
and
the barycenters on those $1$-strata
which are not adjacent to higher dimensional strata.
The partition divides
$\Phi$
into some finite number of connected components.
 
For each $0$-stratum
$P$,
there is a unique connected component containing
$P$.
Take a suitable slight enlargement of rounding of corners of its closure.
We denote by
$\widetilde{\bfW}(P)$
and
$\del^2 \widetilde{\bfW}(P)$
the inverse image under
$\bar{\pi}$
respectively
of the slight enlargement
and
of its boundary.
Perturbing these boundaries if necessary,
we may assume that
$\del^2 \widetilde{\bfW}(P), \del^2 \widetilde{\bfW}(P^\prime)$
are
disjoint
or
intersect transversely
for any pair
$P, P^\prime$
of $0$-strata.
Perturb each edge connecting the barycenter of a top dimensional stratum
$S^{(n)}$
to that of a $1$-dimensional lower stratum
$S^{(n-1)}$
so that
it becomes locally in
$\Phi_{n-1}$
perpendicular to
$S^{(n-1)}$.
More precisely,
the edge becomes parallel to the normal direction in the handle
$\Sigma_{S^{(n-1)}} \times S^{(n-1)}
\subset
\Phi_{n-1} \setminus \Phi_{n-2}$
to 
$S^{(n-1)}$. 
This is always possible
since
$\Phi$
is smoothly normally conical.
Repeat the process inductively until reaching the barycenters of $1$-strata in
$\Phi_1$.
Perform the same perturbation to each edge starting from the barycenter of a lower dimensional closed stratum
which is not adjacent to higher dimensional strata.
We use the same symbols
$\widetilde{\bfW}(P), \del^2 \widetilde{\bfW}(P)$
to denote the results.

\begin{remark}
Our construction works even
when
$\Phi$
is not closed.
Then for each stratum we will take the barycenter of its closure,
which makes sense due to the condition
$(iii)$.
Moreover,
for instance,
we will have connected components not adjacent to vertices,
which only causes irrelevant complication in labeling.
\end{remark}

\begin{lemma} \label{lem:W-sector}
Each
$\widetilde{\bfW}(P)$
is a Weinstein sector.
\end{lemma}
\begin{proof}
We adapt the idea in
\cite[Example 1.33]{GPS2}.
Recall that
we are assuming
$\Phi$
to be closed.
For simplicity,
we will further assume that
there is no lower dimensional strata
which are not adjacent to higher dimensional strata.
Since by construction
$\bar{\pi}(\del^2 \widetilde{\bfW}(P))$
is away from
$P$,
we may assume that
$\del^2 \widetilde{\bfW}(P)$
is disjoint from
$T^* \widehat{M}_P
\subset
\widetilde{\bfW}(P)
\cap
\widetilde{\bfW}(\Phi_0)$.
Let
$I$
be a $1$-stratum
which intersects
$\bar{\pi}(\del^2 \widetilde{\bfW}(P))$
at a point
$p$.
Due to the above perturbation,
locally in
$\Phi_1$
the image
$\bar{\pi}(\del^2 \widetilde{\bfW}(P))$
becomes perpendicular to
$I$.
Since
$\bar{\pi}$
kills
the base direction of
$T^* \widehat{M}_I$
and
the fiber direction of
$T^* I_\circ$,
over the neighborhood of
$I$
the boundary
$\del^2 \widetilde{\bfW}(P)$
is isomorphic to a product of
$T^* \widehat{M}_I$
and
the fiber
$T^*_p I_\circ$.
Inductively,
let
$S^{(k)}$
be a $k$-stratum adjacent to a $(k-1)$-stratum
$S^{(k-1)}$
which intersects
$\bar{\pi}(\del^2 \widetilde{\bfW}(P))$.
Due to the above perturbation,
locally in
$\Phi_k$
the image
$\bar{\pi}(\del^2 \widetilde{\bfW}(P))$
becomes perpendicular to
$S^{(k)}$.
Since
$\bar{\pi}$
kills
the base direction of
$T^* \widehat{M}_{S^{(k)}}$
and
the fiber direction of
$T^* S^{(k)}_\circ$,
over the neighborhood of
$S^{(k)}$
the boundary
$\del^2 \widetilde{\bfW}(P)$
is isomorphic to a product of
$T^* \widehat{M}_{S^{(k)}}$
and
the restriction
$T^* S^{(k)}_\circ |_{\bar{\pi}(\del^2 \widetilde{\bfW}(P)) \cap S^{(k)}}$.
By construction locally in
$\Phi_{(k-1)}$
the intersection
$\bar{\pi}(\del^2 \widetilde{\bfW}(P)) \cap S^{(k)}$
becomes perpendicular to
$S^{(k-1)}$,
the handle attachment
$\bfW(\Phi_{k-1}) \#_{\cL_{k-1}} (T^* \widehat{M}_{S^{(k)}} \times T^* S^{(k)}_\circ)$
glues this product to the gluing of the products in the previous steps.

Rewinding the process,
one sees that
$\del^2 \widetilde{\bfW}(P)$
is obtained by extending via the handle attachments the union of the inverse images under
$\bar{\pi}$
of hypersurfaces in the intersection of
$\Phi_n \setminus \Phi_{n-1}$
and
the connected component containing
$P$.
We claim that
this hypersurface
$\del^2 \widetilde{\bfW}(P)$
is sectorial.
Since the newly attached pieces are cylindrical,
so is
$\del^2 \widetilde{\bfW}(P)$.
It remains to show the existence of an $\alpha$-defining function on
$\del^2 \widetilde{\bfW}(P)$.  
By
\cite[Construction 12.17]{GPS2}
we may assume
$\del^2 \widetilde{\bfW}(P)$
to be one of the hypersurfaces-with-corners dividing
$\widetilde{\bfW}(\Phi)$
into the original connected components.
Let
$I$
be a $1$-stratum adjacent to
$P$
and
$H_{P, I}$
any of such hypersurfaces intersecting
$I$.
Fix the standard coordinates
$(x, y)$
on the symplectic manifold
$T^* I_\circ$.
Unwinding
the filtration of
$\Phi$
and
the construction of
$\widetilde{\bfW}(\Phi)$,
one sees that
$H_{P, I}$
is given by
$\{ x = 0 \}$
up to perturbation.
We denote by
$y_{P, I}$
the function on a neighborhood of
$H_{P, I}$
which is canonically induced by
$y \colon T^* I_\circ \to \bR$.
Clearly,
$y_{P, I}$
is linear near infinity.
The Hamiltonian vector field
$X_{y_{P, I}}$
is outward pointing along
$H_{P, I}$,
as it is given by a lift of
$\del_x$
on 
$T^* I_\circ$.
Hence
$H_{P, I}$
is sectorial.
Thus the Weinstein manifold-with-boundary
$\widetilde{\bfW}(P)$
satisfies the condition in
\cite[Definition 2.4]{GPS1}
to be a Liouville sector.
\end{proof}

\begin{remark}
Suppose that
there is a lower dimensional closed stratum
which is not adjacent to higher dimensional strata.
Let
$S^{(k)}$
be such a stratum.
Then for each $0$-stratum
$P$
adjacent to it,
$\del^2 \widetilde{\bfW}(P)$
gets contribution from the extensions via the handle attachments of the inverse images under
$\bar{\pi}$
of hypersurfaces in
$S^{(k)}$.
\end{remark}

\begin{remark}
Suppose that
$\Phi$
is not closed
and,
for instance,
there is a half-open $1$-stratum
$I$
adjacent to
$P$.
Recall that
in this case the handle attachment
$\bfW(\Phi_1) \#_{\cL_1} (T^* \widehat{M}_I \times T^* I_\circ)$
is not precisely a Liouville manifold-with-boundary,
as one needs to modify the canonical Liouville form of
$T^* I_\circ$
near
$\del T^* I_\circ$
making it convex.
This modification,
which is not necessary if
$I$
were closed,
yields a stop on the newly formed ideal boundary of the handle attachment.
\end{remark}

\begin{lemma} \label{lem:W-SC}
The union
$\bigcup_{P} \widetilde{\bfW}(P)$
gives a Weinstein sectorial cover of the modification
$\widetilde{\bfW}(\Phi)$.
\end{lemma}
\begin{proof}
We adapt the idea in
\cite[Example 1.34]{GPS2}.
Again,
assume that
$\Phi$
is closed
and
there is no lower dimensional strata
which are not adjacent to higher dimensional strata.
Finiteness of the cover
$\widetilde{\bfW}(\Phi)
=
\bigcup_{P} \widetilde{\bfW}(P)$
follows from the condition
$(i)$.
By construction each
$\widetilde{\bfW}(P)$
has no corners
and
$\del \widetilde{\bfW}(\Phi)$
is empty.
Moreover,
$\del^2 \widetilde{\bfW}(P), \del^2 \widetilde{\bfW}(P^\prime)$
intersect transversely for any distinct pair
$P, P^\prime$
of $0$-strata
unless they are disjoint.
By
\pref{lem:W-sector}
each
$\widetilde{\bfW}(P)$
is a Weinstein sector.
Hence the collection of the hypersurfaces 
$\del^2 \widetilde{\bfW}(P)$
is sectorial
if there exist functions
\begin{align*}
\tilde{y}_{P}
\colon
\Nbd^Z_{\widetilde{\bfW}(\Phi)}(\del^2 \widetilde{\bfW}(P))
\to
\bR
\end{align*}
linear near infinity satisfying
\begin{align*}
d \tilde{y}_{P} |_{C_{P^\prime}} = 0 \ \text{for} \ P \neq P^\prime, \
\{ \tilde{y}_{P}, \tilde{y}_{P^\prime} \} = 0
\end{align*}
where
$C_P$
denote the characteristic foliations of
$\del^2 \widetilde{\bfW}(P)$.

Consider the functions
$y_{P, I}$
in the proof of
\pref{lem:W-sector}.
Pushing forward
$X_{y_{P, I}}$
by
$\bar{\pi}$,
we obtain a collection of vector fields on neighborhoods of
$\bar{\pi}(H_{P, I})$.
Here,
by a vector field on a subset of
$\Phi$
we mean the gluing of vector fields on submanifolds of
$\Phi_k \setminus \Phi_{k-1}$
for
$k = 1, \ldots, n$.
Take their deformations around the corners to concatenate
$\bar{\pi}_*(X_{y_{P, I}})$
so that
translation of the result gives a vector field
$\tilde{X}_P$
on a neighborhood
$\bar{\pi}(\del^2 \widetilde{\bfW}(P))$,
which is outward pointing along
$\del^2 \widetilde{\bfW}(P)$.
Recall that
$\bar{\pi}(\del^2 \widetilde{\bfW}(P)),
\bar{\pi}(\del^2 \widetilde{\bfW}(P^\prime))$
intersect transversely for any distinct pair
$P, P^\prime$
of $0$-strata
unless they are disjoint.
Since
$\Phi$
has only finitely many $0$-strata,
we may deform further
$\tilde{X}_P$
simultaneously
so that
they in addition become tangent to
$\del^2 \widetilde{\bfW}(P^\prime)$
on the intersections.
We use the same symbol
$\tilde{X}_P$
to denote these deformations.
We define
$\tilde{y}_{P}$
as the Hamiltonian lifts of
$\tilde{X}_P$,
which make sense due to the construction of
$\widetilde{\bfW}(\Phi)$
and
our definition of a vector field on a subset of
$\Phi$.

It remains to show that
the sectorial cover is Weinstein.
All strata
\begin{align*}
\widetilde{\bfW}(\Phi)_{J, K, L}
=
\bigcup_{P \in J} \widetilde{\bfW}(P)
\cap
\bigcap_{P^\prime \in K} \del^2 \widetilde{\bfW}(P^\prime)
\setminus
\bigcap_{P^{\prime \prime} \in L} \widetilde{\bfW}(P^{\prime \prime})
\end{align*}
are the inverse images under
$\bar{\pi}$
of
\begin{align*}
\bar{\pi}(\widetilde{\bfW}(\Phi))_{J, K, L}
=
\bigcup_{P \in J} \bar{\pi}(\widetilde{\bfW}(P))
\cap
\bigcap_{P^\prime \in K} \bar{\pi}(\del^2 \widetilde{\bfW}(P^\prime))
\setminus
\bigcap_{P^{\prime \prime} \in L} \bar{\pi}(\widetilde{\bfW}(P^{\prime \prime}))
\end{align*}
Recall that
locally in
$\Phi_k$
the images
$\bar{\pi}(\del^2 \widetilde{\bfW}(P))$
are perpendicular to
$S^{(k)}$.
Hence over
$\Phi_k \setminus \Phi_{k-1}$
the convex completions of the symplectic reductions of
$\widetilde{\bfW}(\Phi)_{J, K, L}$
are given by the union of products of
$T^* \widehat{M}_{S^{(k)}}$
and
the restrictions
$T^* S^{(k)}_\circ |_{\bar{\pi}(\widetilde{\bfW}(\Phi))_{J, K, L} \cap S^{(k)}}$.
In particular,
over
$\Phi_n \setminus \Phi_{n-1}$
the convex completions of the symplectic reductions of
$\widetilde{\bfW}(\Phi)_{J, K, L}$
are simply given by the cotangent bundles of
$\bar{\pi}(\widetilde{\bfW}(\Phi))_{J, K, L}$.
These unions for
$k = 1, \ldots, n$
are
Weinstein up to deformation
and
glued to yield the convex completion of the symplectic reduction of
$\widetilde{\bfW}(\Phi)_{J, K, L}$.
\end{proof}

\begin{remark}
The construction does not work without modification of
$\widetilde{\bfW}(\Phi)$,
since in general
$\del \widetilde{\bfW}(\Phi)$
is nonempty
and
intersects
$\del^2 \widetilde{\bfW}(P)$.
\end{remark}

\subsection{Sectorial descent and HMS over fanifolds}
Our Weinstein sectorial cover is compatible with HMS established by Gammage--Shende in
\cite{GS1}.
On the $A$-side,
there are three isomorphic constructible cosheaves of categories over
$\Phi$.
First,
consider the functor
\begin{align*}
\fsh_*
\colon
(\Fan^\twoheadrightarrow)^{op}
\to
^{**} \DG, \
\Sigma
\mapsto
\Sh_{\bL(\Sigma)}(T_\Sigma)
\end{align*}
which is obtained from
$\fsh
\colon
\Fan^\twoheadrightarrow
\to
^{*} \DG^*$
in
\cite[Section 4.5]{GS1}
by taking adjoints of the images of the morphisms.
Here,
$^{**} \DG$
denotes the category of
cocomplete dg categories
and
functors
which preserves
colimits
and
compact objects,
while
$^{*} \DG^*$
denotes the category of
cocomplete dg categories
and
functors
which preserves 
limits
and
colimits.
Note that
we have
$^{*} \DG^* = (^{**} \DG)^{op}$.
We use the same symbol
$\fsh_*$
to denote the composition
$\Exit(\Phi)^{op}
\to
^{**} \DG$
with the functor
$\Exit(\Phi)^{op}
\to
(\Fan^\twoheadrightarrow)^{op}$
equipped with
$\Phi$.

Second,
consider the functor
\begin{align*}
\pi_* \msh_{\widetilde{\bL}(\Phi) *}
\colon
\Exit(\Phi)^{op}
\to
^{**} \DG
\end{align*}
which is obtained from
$\pi_* \msh_{\widetilde{\bL}(\Phi)}
\colon
\Exit(\Phi)
\to
^{*} \DG^*$
in
\cite[Section 4.5]{GS1}
by taking adjoints of the images of the morphisms.
By definition both
$\fsh_*$
and
$\pi_* \msh_{\widetilde{\bL}(\Phi) *}$
are constructible cosheaves of categories over
$\Phi$.
From
\cite[Proposition 4.34]{GS1}
we obtain an isomorphism
\begin{align*}
\fsh_*
\cong
\pi_* \msh_{\widetilde{\bL}(\Phi) *}.
\end{align*}

Third,
let
$\UI(\Phi)$
be the category of unions of intersections of
$\widetilde{\bfW}(P)$
in the Weinstein sectorial cover
$\widetilde{\bfW}(\Phi)
=
\bigcup_P \widetilde{\bfW}(P)$
from
\pref{lem:W-SC}.
Morphisms are given by inclusions of Weinstein sectors.
Introduce to
$\UI(\Phi)$
a topology generated by maximal proper Weinstein subsectors
$\widetilde{\bfW}(P)$
of
$\widetilde{\bfW}(\Phi)$
in the cover.
Consider the functor
$\Fuk_*
\colon
\UI(\Phi)
\to
^{**} \DG$
which sends
each Weinstein sector to the category of modules over the wrapped Fukaya category.
The images of morphisms are induced by the pushforward functors from
\cite{GPS1}
along the inclusions of Weinstein sectors.
Due to
the covariant functoriality
and
sectorial descent
of wrapped Fukaya categories,
$\Fuk_*$
gives a cosheaf of categories over
$\UI(\Phi)$.
Note that
$\Fuk_*$
preserves compact objects
as it is defined before taking module categories.
Since
$\bar{\pi}$
restricts to
$\pi$
and
$\widetilde{\bfW}(\Phi)
=
\bigcup_P \widetilde{\bfW}(P)$
is a lift of an open cover of
$\widetilde{\bL}(\Phi)$,
the pushforward
\begin{align*}
\bar{\pi}_* \Fuk_*
\colon
\Exit(\Phi)^{op}
\to
^{**} \DG
\end{align*}
gives a constructible cosheaf of categories over
$\Phi$.

\begin{lemma} \label{lem:topological}
There is an isomorphism
$\pi_* \msh^{op}_{\widetilde{\bL}(\Phi) *}
\cong
\bar{\pi}_* \Fuk_*$.
\end{lemma}
\begin{proof}
Over a sufficiently small open neighborhood of each intersection of
$\bar{\pi}(\widetilde{\bfW}(P))$,
the inverse image under
$\bar{\pi}$,
regarded as a stopped Weinstein sector as in Remark
\pref{rmk:2.3},
has the inverse image under
$\pi$
as its relative skeleton.
Hence by
\cite[Theorem 1.4]{GPS3}
the sections over such open neighborhoods are canonically equivalent.
Moreover,
the corestriction functors between sections of the left
and
the pushforward functors between sections of the right
intertwine these canonical equivalences. 
\end{proof}

On the $B$-side,
the composition
\begin{align*}
\Coh_! \circ \bfT
\colon
\Exit(\Phi)^{op}
\to
^{**} \DG, \
S^{(k)}
\mapsto
\Coh(\bfT(S^{(k)}))
\end{align*}
of
$\Coh_!$
with
$\bfT$
from
\cite[Section 3]{GS1}
gives a constructible cosheaves of categories over
$\Phi$.
Recall that
$\bfT(S^{(k)})$
is the toric variety for the fan
$\Sigma_{S^{(k)}} \subset M_{S^{(k)}, \bR}$,
where
$(\Sigma_{S^{(k)}}, M_{S^{(k)}})$
is the image of
$S^{(k)}$
under the functor
$\Exit(\Phi)^{op}
\to
(\Fan^\twoheadrightarrow)^{op}$
equipped with
$\Phi$.
Recall also that
the values of
$\Coh_! \circ \bfT$
are the module categories over the dg categories of coherent sheaves on toric varieties
and
the images of morphisms are the pushforwards along closed immersions.
By
\cite[Proposition 3.10]{GS1} the colimit
$\bfT(\Phi)
=
\colim_{\Exit(\Phi)^{op}} \bfT(S^{(k)})$
exists as an algebraic space.
Moreover,
the canonical functor
\begin{align*}
\Coh_! \circ \bfT(\Phi)
=
\colim_{\Exit(\Phi)^{op}} \Coh \bfT(S^{(k)})
\to
\Coh \left( \colim_{\Exit(\Phi)^{op}} \bfT(S^{(k)}) \right)
=
\Coh(\bfT(\Phi))
\end{align*}
is an equivalence in
\cite[Proposition 3.14]{GS1}.

\begin{theorem}[cf. {\cite[Corollary 5.8]{GS1}}] \label{thm:HMS}
There is an isomorphism
$\bar{\pi}_* \Fuk_*
\cong
\Coh_! \circ \bfT$.
\end{theorem}
\begin{proof}
From
\cite[Theorem 5.3]{GS1}
we obtain an isomorphism
$\pi_* \msh^{op}_{\widetilde{\bL}(\Phi) *}
\cong
\Coh_! \circ \bfT$.
Combined with
\pref{lem:topological},
it yields the desired isomorphism.
\end{proof}

Taking the global section,
we recover the equivalence
\cite[Theorem 5.4]{GS1}
\begin{align} \label{eq:HMS}
\Fuk(\widetilde{\bfW}(\Phi))
\simeq
\Coh(\bfT(\Phi)).
\end{align}

\section{HMS for very affine hypersurfaces}
In this section,
we revisit HMS for very affine hypersurfaces established by Gammage--Shende.
Specifically,
we describe an aspect of the mirror pair as the associated spaces with a fanifold.
Then we apply our results in the previous section,
giving an affirmative answer to the conjecture raised in
\cite[Section 2.5]{GS2}.
As a consequence,
we promote HMS for very affine hypersurfaces to an isomorphism of constructible cosheaves of categories over the fanifold. 

\subsection{Mirror pair for very affine hypersurfaces}
Let
$H
\subset
M^\vee_\bC
\cong
T^* T^{d+1}
\cong
(\bC^*)^{d+1}$
be a very affine hypersurface. 
Assume that,
via the standard correspondence,
its defining Laurent polynomial determines a smooth quasiprojective stacky fan
$\Sigma \subset M_\bR \cong \bR^{d+1}$.
Assume further that
its stacky primitives span a convex lattice polytope
$\Delta^\vee$
containing the origin
and
$\Sigma$
defines an adapted star-shaped triangulation
$\cT$
of
$\Delta^\vee$.
Let
$\bfT_\Sigma$
be the toric stack for
$\Sigma$.
By
\cite[Theorem 1.1]{GS2}
for its toric boundary divisor
$\del \bfT_\Sigma$
there is an equivalence
\begin{align} \label{eq:GS}
\Fuk(H) \simeq \Coh(\del \bfT_\Sigma).
\end{align}

Below,
we briefly review the proof.
Due to the assumptions,
one can construct a global skeleton
$\Core(H)$
of
$H$.
Recall that
we restrict to
$H$
the canonical Weinstein structure on
$M^\vee_\bC \cong T^* T^{d+1}$.  
In the sequel,
for simplicity we will assume that
$\Sigma$
is an ordinary simplicial fan.
When 
$\Sigma$
is a smooth quasiprojective stacky fan,
we will replace
the $A$-model with its finite cover
and
the $B$-model with the associated toric stack.
The skeleton
$\Core(H)$
can be obtained by gluing skeleta of tailored pants
$\tilde{P}_d$
computed by Nadler
\cite{Nad},
after transported to the ones in the pants decomposition of
$H$.
As explained in
\cite{Nad},
we may identify the skeleton
$\Core(\tilde{P}_d)$
with the ideal boundary of the FLTZ skeleton associated with the toric variety
$\bA^{d+1}$.
Hence we may identify
$\Core(H)$
with
$\del_\infty \bL(\Sigma)$
for the FLTZ skeleton
$\bL(\Sigma)$
associated with
$\bfT_\Sigma$.
Moreover,
by
\cite[Corollary 4.8]{GS2}
the skeleton
$\Core(H) = \del_\infty \bL(\Sigma)$
has an open cover by the subsets
$\bL(\Sigma / \sigma) \times \del_\infty \sigma$
anti-indexed by the poset of nonzero cones in
$\Sigma$.

To each
$\sigma \in \Sigma$,
let
$\overline{O(\sigma)}$
be the closure of the corresponding toric orbit
$O(\sigma)$.
It is the toric variety for
$\Sigma / \sigma$.
Here,
we take its reduced induced scheme structure.
If
$\sigma \subset \tau$,
then there is a closed immersion
$\overline{O(\tau)}
\hookrightarrow
\overline{O(\sigma)}$
induced by the quotient
$\Sigma / \sigma
\to
\Sigma / \tau$.
Also,
we have
$\overline{O(\sigma)}
\cap
\overline{O(\sigma^\prime)}
=
\overline{O(\sigma \wedge \sigma^\prime)}$
where
$\sigma \wedge \sigma^\prime$
denotes the smallest cone in
$\Sigma$
containing both
$\sigma$
and
$\sigma^\prime$
if such a cone exists,
otherwise
$\emptyset$.
By definition
$\del \bfT_\Sigma$
is the union of the nontrivial toric orbit closures.
Scheme-theoretically,
this amounts to the existence of the colimit
\begin{align} \label{eq:5.3}
\del \bfT_\Sigma
=
\colim_{0 \neq \sigma \in \Sigma}\overline{O(\sigma)}
\end{align}
as an algebraic space
\cite[Lemma 3.5]{GS2}.
Then we have
\begin{align*}
\Coh(\del \bfT_\Sigma)
=
\Coh(\colim_{0 \neq \sigma \in \Sigma}\overline{O(\sigma)})
=
\colim_{0 \neq \sigma \in \Sigma}\Coh(\overline{O(\sigma)}).
\end{align*}
By
\cite[Theorem 1.4]{GPS3}
and
\cite[Corollary 4.8]{GS2}
we have
\begin{align*}
\Fuk(H)
=
\msh_{\bL(\Sigma)}(\bL(\Sigma))^{op}
=
\colim_{0 \neq \sigma \in \Sigma} \msh_{\bL(\Sigma)}(\sigma^\perp \times \del_\infty \sigma)^{op}.
\end{align*}
As explained in the proof of
\cite[Theorem 7.13]{GS2},
we have 
\begin{align*}
\colim_{0 \neq \sigma \in \Sigma} \msh_{\bL(\Sigma)}(\sigma^\perp \times \del_\infty \sigma)^{op}
=
\colim_{0 \neq \sigma \in \Sigma}\Coh(\overline{O(\sigma)}).
\end{align*}
Combining these three equivalences,
one obtains
\pref{eq:GS}.

\subsection{Mirror pair over a fanifold}
Up to tailoring
$H$,
the mirror pair
$(H, \del \bfT_\Sigma)$
can be obtained as a special case of mirror pairs
$(\widetilde{\bfW}(\Phi), \bfT(\Phi))$
over
$\Phi$.
Let
$\Phi_\Sigma
=
\Sigma \cap S^d
\subset
S^d$
be the fanifold from
\cite[Example 2.11, 4.23]{GS1}
where
$S^d
=
\{ \Sigma^{d+1}_{i = 1} x^2_i = 1 \}
\subset
\bR^{d+1}$
is the standard embedding of the $d$-sphere.
For instance,
any vertex
$P$
is the intersection of a ray
$\rho_P \in \Sigma(1)$
and
$S^d$.
To it one associates the pair
$(\Sigma_P, M_P)
=
(\Sigma / \rho_P, M / \langle \rho_P \rangle)$.
Clearly,
$\Phi_\Sigma$
is closed
and
there is no lower dimensional strata
which are not adjacent to higher dimensional strata.

\begin{lemma} \label{lem:A-matching}
There are open neighborhoods of
$\Core(H)$
in
$H$
and
$\widetilde{\bL}(\Phi_\Sigma)$
in
$\widetilde{\bfW}(\Phi_\Sigma)$
which are isomorphic as a Weinstein manifold.
\end{lemma}
\begin{proof}
For each top dimensional cone
$\sigma \in \Sigma_{\max}$,
consider the subfanifold
$\Phi_\sigma = \bar{\sigma} \cap S^d$
of
$\Phi_\Sigma$.
Unwinding the proof of
\cite[Theorem 5.13]{Nad},
one sees that
the associated Weinstein sector
$\widetilde{\bfW}(\Phi_\sigma)$
can be identified with the open neighborhood
$U_d$
there of
$\Core(\tilde{P}_d)$
in
$\tilde{P}_d$,
after we canonically transform
$\tilde{P}_d$
to the pants 
whose tropicalization is dual to
$\sigma$.
Hence,
up to suitable transformation,
we may regard
$\widetilde{\bfW}(\Phi_\Sigma)$
as the union of
$U_d$
and
$\widetilde{\bL}(\Phi_\Sigma)$
as the union of
$\Core(\tilde{P}_d)$.
Then,
near their skeleta,
$H$
and
$\widetilde{\bfW}(\Phi_\Sigma)$
are symplectomorphic by the same argument as in
\cite[Theorem 5.13]{Nad}.
Since over
$\Phi_\Sigma$
the Weinstein handle attachments can be carried out in the ambient cotangent bundle
$M^\vee_\bC \cong T^* T^{d+1}$,
the Weinstein structure on
$\widetilde{\bfW}(\Phi_\Sigma)$
coincides with the restriction of the canonical one on
$M^\vee_\bC$.
\end{proof}

\begin{lemma} \label{lem:B-matching}
The algebraic space
$\bfT(\Phi_\Sigma)$
coincides with
$\del \bfT_\Sigma$.
\end{lemma}
\begin{proof}
By
\cite[Proposition 3.13]{GS1}
we have
$\bfT(\Phi_\Sigma)
=
\colim_{\Exit(\Phi_\Sigma)^{op}} \bfT(S^{(k)})$.
As one associates the pair
$(\Sigma_P, M_P)
=
(\Sigma / \rho_P, M / \langle \rho_P \rangle)$
to each vertex
$P \in \Phi_\Sigma$,
the toric variety
$\bfT(P)$
coincides with
$\overline{O(\rho_P)}$.
Let
$\sigma_{S^{(k)}} \in \Sigma_P$
be the cone corresponding to an exit path in
$\Phi_\Sigma$
from
$P$
to
$S^{(k)}$.
Then by definition of fanifolds
$\bfT(S^{(k)})$
is the toric variety for the fan
$\Sigma_P / \sigma_{S^{(k)}}
\subset
M_\bR / \langle \sigma_{S^{(k)}} \rangle$,
which in turn is the toric orbit closure corresponding to the inverse image
$\tilde{\sigma}_{S^{(k)}}$
of
$\sigma_{S^{(k)}}$
under the quotient
$\Sigma \to \Sigma_P$. 
When
$S^{(k)}$
runs through the strata of
$\Phi_\Sigma$,
the inverse images
$\tilde{\sigma}_{S^{(k)}}$
range over nonzero cones in
$\Sigma$.
Thus from
\pref{eq:5.3}
we obtain
$\colim_{\Exit(\Phi_\Sigma)^{op}} \bfT(S^{(k)})
=
\del \bfT_\Sigma$.
\end{proof}

\subsection{Sectorial descent and HMS for very affine hypersurfaces}
By
\pref{lem:A-matching}
and
\pref{lem:B-matching}
the equivalence
\pref{eq:GS}
coincides with
\pref{eq:HMS}
for
$\Phi_\Sigma$,
i.e.
$\Fuk(\widetilde{\bfW}(\Phi_\Sigma)) \simeq \Coh(\bfT(\Phi_\Sigma))$.
Finally,
we give an affirmative answer to the conjecture raised in
\cite[Section 2.5]{GS2}.

\begin{theorem}
There is a Weinstein sectorial cover of
$H$
lifting the open cover of
$\Core(H) = \del_\infty \bL(\Phi_\Sigma)$
from
\cite[Corollary 4.8]{GS2}.
\end{theorem}
\begin{proof}
Each piece
$\widetilde{\bfW}(P)$
of our Weinstein sectorial cover
$\widetilde{\bfW}(\Phi_\Sigma)
=
\bigcup_P \widetilde{\bfW}(P)$
is a Weinstein sector with skeleton
$\bL(\Sigma / \rho_P)$
obtained from
$T^* \widehat{M}_P$
for a vertex
$P \in \Phi_\Sigma$
by the inductive Weinstein handle attachments.
Intersections of
$\widetilde{\bfW}(P)$
are Weinstein sectors-with-sectorial-corners
whose skeleta are given by intersections of
$\bL(\Sigma / \rho_P)$.
Since by
\cite[Corollary 4.8]{GS2}
the union of
$\bL(\Sigma / \rho_P)
\cong
\bL(\Sigma / \rho_P) \times \del_\infty \rho_P$
covers
$\del_\infty \bL(\Sigma)$
when
$\rho_P$
runs through the rays in
$\Sigma$,
the Weinstein sectorial cover
$\widetilde{\bfW}(\Phi_\Sigma)
=
\bigcup_P \widetilde{\bfW}(P)$
gives the desired lift.
\end{proof}



\end{document}